\documentclass[11pt]{amsart}

\usepackage[T1]{fontenc}
\usepackage[utf8]{inputenc}
\usepackage{lmodern}
\usepackage{microtype}
\usepackage{amsmath,amssymb,amsthm,mathtools}
\usepackage[hidelinks]{hyperref}

\newtheorem{theorem}{Theorem}[section]
\newtheorem{proposition}[theorem]{Proposition}
\newtheorem{lemma}[theorem]{Lemma}

\theoremstyle{remark}

\newcommand{\K}{\mathcal K}
\newcommand{\HH}{\mathcal H}
\newcommand{\I}{\mathcal I}

\newcommand{\cl}[1]{\overline{#1}}
\newcommand{\down}[1]{\mathord{\downarrow}\!#1}
\newcommand{\up}[1]{\mathord{\uparrow}\!#1}

\title[Comeager hereditary families]{Comeager hereditary families of compact sets are big}

\author{Peter Banáš}

\address{Faculty of Mathematics and Physics, Charles University, Prague, Czech Republic}
\email{peter.banas149@student.cuni.cz, p.banas18@gmail.com}


\begin{document}

\begin{abstract}
Let $X$ be a Polish space and let $\K(X)$ be its Vietoris hyperspace.
A family $\I\subseteq\K(X)$ is hereditary if it is downward closed under
inclusion. Matheron and Zelený asked whether every comeager hereditary
family in $\K(X)$ contains a dense hereditary $G_\delta$ subfamily.
We give an affirmative answer in ZFC.
\end{abstract}

\maketitle

\section{Introduction}

Let $X$ be a Polish space, and let $\K(X)$ denote the space of all compact
subsets of $X$ equipped with the Vietoris topology.
A family $\I\subseteq\K(X)$ is \emph{hereditary} if
\[
K\subseteq L\in\I \quad\Longrightarrow\quad K\in\I.
\]
Following Matheron and Zelený~\cite{MZ05}, a set
$\mathcal A\subseteq\K(X)$ is called \emph{big} if it contains a dense
hereditary $G_\delta$ subset of $\K(X)$.

Matheron and Zelený asked whether every comeager hereditary family is big
\cite[Problem~2.1]{MZ05}. 

Notice that every comeager family $\I \subseteq\K(X)$ contains a dense $G_\delta$ subset of $\K(X)$. Thus, for a comeager hereditary family, the issue is whether such a subset can be chosen hereditary.

The question was later reiterated in their survey
\cite[Problem~6.9]{MZ07}. They proved the assertion for coanalytic families
\cite[Theorem~2.2]{MZ05}; see also \cite[Theorem~5.11]{MZ07}. Without a
definability assumption, they obtained the conclusion under the additional
set-theoretic hypothesis
\[
\omega_1^{L[x]}<\omega_1
\qquad\text{for every }x\in\omega^\omega
\]
\cite[Theorem~2.9]{MZ05}; see also \cite[Remark~5.12]{MZ07}.
Our main result shows that no additional hypothesis is needed.

\begin{theorem}\label{thm:main}
Let $X$ be a Polish space. Every comeager hereditary family
$\I\subseteq\K(X)$ contains a dense hereditary
$G_\delta$ subset of $\K(X)$.
\end{theorem}

We recall the ideal-theoretic criterion from \cite{MZ05} that will be used to
prove the theorem. For $n\geq1$, put
\[
P(n)=\{K\in\K(X):|K|\leq n\}.
\]
Each $P(n)$ is closed in $\K(X)$. Following \cite[Definition~2.3]{MZ05}, let $\HH$ be the family of all closed
sets $F\subseteq\K(X)$ such that
\[
F\cap P(n)\text{ is nowhere dense in }P(n)
\qquad\text{for every }n\geq1.
\]
Let
\[
\HH^{\mathrm{ext}}
=
\left\{
B\subseteq\K(X):
B\subseteq\bigcup_{m\in\mathbb N}F_m
\text{ for some sequence }(F_m)_{m\in\mathbb N}\subseteq\HH
\right\}.
\]
Thus $\HH^{\mathrm{ext}}$ is the $\sigma$-ideal generated by $\HH$.
The following characterization is \cite[Corollary~2.5]{MZ05}.

\begin{theorem}[Matheron--Zelený]\label{thm:MZcriterion}
For every $\mathcal A\subseteq\K(X)$,
\[
\mathcal A\text{ is big}
\quad\Longleftrightarrow\quad
\K(X)\setminus\mathcal A\in\HH^{\mathrm{ext}}.
\]
\end{theorem}

\section{Preliminary lemmas}

For $A\subseteq\K(X)$, define its hereditary closure by
\[
\down A
=
\{K\in\K(X): (\exists L \in A) (K\subseteq L)\}.
\]
For $m\geq1$ and open sets $U_1,\ldots,U_m\subseteq X$, let
$$\langle U_1,\ldots,U_m\rangle=
\left\{
M\in\K(X):
M\subseteq\bigcup_{i=1}^m U_i
\text{ and }
M\cap U_i\neq\varnothing\text{ for all }i
\right\}.$$
Together with
$\{\varnothing\}$, these sets form a basis for the Vietoris topology on $\K(X)$.
\begin{lemma}\label{lem:open}
If $A\subseteq\K(X)$ is open, then $\down A$ is open.
\end{lemma}

\begin{proof}
Let $K\in\down A$. If $K=\varnothing$, then $\{\varnothing\}$ is an
open neighbourhood of $K$ contained in $\down A$. Otherwise choose
$L\in A$ with $K\subseteq L$ and a basic Vietoris neighbourhood
$L\in\langle U_1,\ldots,U_m\rangle \subseteq A.$ Put
$U=\bigcup_{i=1}^m U_i$. For every compact $M\subseteq U$, the union
$M\cup L$ is contained in $U$ and meets each $U_i$, so
\[
M\subseteq M\cup L\in\langle U_1,\ldots,U_m\rangle \subseteq A.
\]
The open set $\{M\in\K(X):M\subseteq U\}$ contains $K$ and lies in $\down A$.
\end{proof}

Fix a bounded compatible complete metric $d$ on $X$ with $d \leq 1$, and equip the nonempty compact subsets of $X$ with the
associated Hausdorff metric. Extend it to $\K(X)$ by declaring
$d_H(\varnothing,K)=1$ whenever $K\neq\varnothing$. This is a compatible
complete metric for the Vietoris topology.

The next elementary fact is the sequential form of
\cite[Lemma~2.6(i)]{MZ05}.

\begin{lemma}\label{lem:finite-approx}
Let $K,L,L_j\in\K(X)$, where $K\subseteq L$, $K\in P(n)$, and
$L_j\to L$. Then there are $K_j\subseteq L_j$ such that
$K_j\in P(n)$ and $K_j\to K$.
\end{lemma}

\begin{proof}
If $K=\varnothing$, take $K_j=\varnothing$. Otherwise write
$K=\{x_1,\ldots,x_m\}$, where $m\leq n$. Since $L\neq\varnothing$,
the sets $L_j$ are nonempty for all sufficiently large $j$. For such
$j$, choose $x_i^j\in L_j$ so that
\[
d(x_i,x_i^j)\leq d_H(L,L_j)+\frac1j
\qquad (1\leq i\leq m),
\]
and put $K_j=\{x_1^j,\ldots,x_m^j\}$. For the finitely many remaining
indices, put $K_j=\varnothing$. Then $K_j\subseteq L_j$,
$|K_j|\leq n$, and $d_H(K_j,K)\to0$.
\end{proof}

\begin{lemma}\label{lem:closure}
If $A\subseteq\K(X)$ is hereditary, then, for every $n\geq1$,
\begin{equation}\label{eq:closure}
P(n)\cap\cl A=\cl{A\cap P(n)}.
\end{equation}
All closures in \eqref{eq:closure} are taken in $\K(X)$.
\end{lemma}

\begin{proof}
Since $P(n)$ is closed, the inclusion from right to left is immediate.

For the reverse inclusion, let $K\in P(n)\cap\cl A$. Choose $L_j \in A$ such that $L_j \to K $. Then, by Lemma~\ref{lem:finite-approx} with $L=K$, we obtain $K_j \subseteq L_j$ such that $K_j \in P(n)$ and $K_j \to K$. By heredity, $K_j \in A$, so $K_j \in A \cap P(n)$ and $K \in \overline{A \cap P(n)}$.
\end{proof}

\begin{proposition}\label{prop:local}
Let $O\subseteq\K(X)$ be open and let $\mathcal U\subseteq\K(X)$ be
dense and open. Then
\[
\cl{\down O\setminus\down{(O\cap\mathcal U)}}\in\HH.
\]
\end{proposition}

\begin{proof}
Put $A=\down O$, $B=\down{(O\cap\mathcal U)}$, and $C=A\setminus B$.
Both $A$ and $B$ are hereditary and, by Lemma~\ref{lem:open}, open.

We first show that, for every $n\geq1$,
\begin{equation}\label{eq:density}
B\cap P(n)\text{ is dense in }A\cap P(n).
\end{equation}
Fix $K\in A\cap P(n)$, and choose $L\in O$ with $K\subseteq L$.
Since $\mathcal U$ is dense and $O$ is open, there are
$L_j\in O\cap\mathcal U$ with $L_j\to L$. By
Lemma~\ref{lem:finite-approx}, there are $K_j\subseteq L_j$ such that
$K_j\in P(n)$ and $K_j\to K$. Each $K_j$ belongs to $B$, proving
\eqref{eq:density}.

Suppose towards a contradiction that $\cl C\cap P(n)$ has nonempty
interior in $P(n)$. Choose a nonempty relatively open set
$W\subseteq P(n)$ such that
\[
W\subseteq\cl C.
\]
Since $C\subseteq A$, Lemma~\ref{lem:closure} gives
\[
W\subseteq P(n)\cap\cl A=\cl{A\cap P(n)}.
\]
Since $W$ is nonempty and relatively open in $P(n)$, it follows that
$W\cap A\cap P(n)\neq\varnothing$. This intersection is relatively open
in $A\cap P(n)$, so \eqref{eq:density} yields $W\cap B\neq\varnothing$.

On the other hand, $C\cap B=\varnothing$, and $B$ is open. Hence
\[
\cl C\cap B=\varnothing,
\]
contradicting $W\subseteq\cl C$. Thus $\cl C\cap P(n)$ is nowhere
dense in $P(n)$ for every $n\geq1$, and therefore $\cl C\in\HH$.
\end{proof}

\section{Proof of the main theorem}

\begin{proof}[Proof of Theorem~\ref{thm:main}]
Choose dense open sets $\mathcal U_n\subseteq\K(X)$ with
$\bigcap_{n\in\mathbb N}\mathcal U_n\subseteq\I$, and let
$(O_m)_{m\in\mathbb N}$ be a countable basis for $\K(X)$. Put
\[
C_{m,n}=\down{O_m}\setminus\down{(O_m\cap\mathcal U_n)}.
\]
By Proposition~\ref{prop:local}, $\cl{C_{m,n}}\in\HH$ for all $m,n$.
We show that
$$\K(X)\setminus \I \subseteq \bigcup_{m,n \in \mathbb N}C_{m,n}.$$

Fix $K\in\K(X)\setminus\I$.
The upper cone $\up K=\{L\in\K(X):K\subseteq L\}$ contains $K$ and is closed, since $\up K = \bigcap_{x \in K}{\{L \in \K(X): x \in L\}}$. Thus $\up K$ is a nonempty
Baire space. Since $\I$ is hereditary, 
$\up K$ is disjoint from $\I$. The relatively open sets $\mathcal U_n\cap\up K$ thus have
empty intersection, so some $\mathcal U_n\cap\up K$ fails to be dense
in $\up K$. Hence there is an $m$ such that
\[
O_m\cap\up K\neq\varnothing
\qquad\text{and}\qquad
O_m\cap\mathcal U_n\cap\up K=\varnothing.
\]
Consequently $K\in\down{O_m}$ and $K\notin\down{(O_m\cap\mathcal U_n)}$,
so $K\in C_{m,n}$.

The closures $\cl{C_{m,n}}$ therefore cover $\K(X)\setminus\I$, so
$\K(X)\setminus\I\in\HH^{\mathrm{ext}}$. By Theorem~\ref{thm:MZcriterion},
$\I$ is big.
\end{proof}

\section*{Acknowledgements}

I thank Miroslav Zelený for reading the argument and
encouraging its publication.

\end{document}